\documentclass[11pt,letterpaper]{article}
\usepackage[utf8]{inputenc}
\usepackage{amsmath}
\usepackage{amsfonts}
\usepackage{amssymb}

\usepackage{amsthm}

\usepackage{graphicx}     
\usepackage{hyperref} 
\usepackage{url}
\usepackage{amsfonts} 

\usepackage{multicol}

\newtheorem{theorem}{Theorem}

\allowdisplaybreaks

\makeatletter
\@addtoreset{footnote}{page}
\makeatother

\begin{document}

\title{Unusual elementary axiomatizations \\ for abelian groups}

\author{Haydee Jiménez Tafur, Carlos Luque Arias, Yeison Sánchez Rubio\\               
\scriptsize Grupo de Álgebra (Universidad Pedagógica Nacional)\\    
\scriptsize Universidad Nacional de Colombia\\                
\scriptsize jimenezhaydee@gmail.com, luque.ca@gmail.com, yei506@gmail.com}                      

\maketitle

\abstract{ One of the most studied algebraic structures with one operation is the Abelian group, which is defined as a structure whose operation satisfies the associative and commutative properties, has identical element and every element has an inverse element. In this article, we characterize the Abelian groups with other properties and we even reduce it to two the list of properties to be fulfilled by the operation. For this, we make use of properties that, in general, are hardly ever mencioned.}

\section*{Introduction}
\noindent  The axiomatic presentation of mathematical theories allows the selection of different sets of axioms for its development. This choice depends on criteria of economy, elegance, simplicity or pedagogy.

The definition of an abelian group was initially formulated for finite groups (with the axioms of closure, associativity, commutativity and existence of inverses) by Kronecker in 1870 and by Weber for infinite groups in 1893 \cite{waerden}. In 1878 Cayley introduces the notion of abstract group and in 1882 Dick Van presents the first explicit definition of this notion \cite{wussing}. In 1938 Tarski  \cite{tarski} defines an abelian group $(G, +)$ as an associative and commutative quasigroup and characterizes it in terms of subtraction using only two axioms, one that indicates that the subtraction is an operation in $G$ and the other which is a property that includes three variables.  In 1952 Higman and Neumann \cite{higman} give an axiomatization for groups with one axiom in terms of division, using three variables.

In 1981 Neumann \cite{neumann81} proposes another single law in terms of multiplication and inversion, in equational form with four variables.  In 1993 McCune \cite{mccune} presents cha\-rac\-terizations of abelian groups with one axiom that has three or five variables, using computational tools, but in terms of operations such as \{addition and inverse\}, \{double subtraction\}, \{double subtraction, identity\}, \{subtraction, identity\}, \{subtraction, inverse\}, \{double subtraction, inverse\}. 

In all cases in which the groups or abelian groups are characterized in equational form with one axiom, it has an extensive expression and the proofs are intricate. How\-ever, in 1996 McCune and Sands \cite{mcsands} proposed a single law but in implicative form, which is simpler than the equational form, not only in appearance but in the proofs too.  

In the present work we give some characterizations of abelian groups with two elementary axioms whose expressions display an elegant simplicity. The same applies to the proofs which can help to understand this basic algebraic structure.  
 
Algebraic structures can be classified, giving them special names, according to the operations they involve.  We have limited ourselves to consider algebraic structures with one operation  $(G,+)$ with $G$ a nonempty set, called \textit{magma} or \textit{groupoid}. The best known are \textit{Semigroup}, a groupoid with  one associative operation (A); \textit{Monoid},  a semigroup that has neutral element (NE); \textit{Group},  a monoid such that all elements have inverse elements (IN); and \textit{Abelian group}, a commutative, (C), group.  

However, there are other properties  (see \cite{ilse}) such as : for all  $a$,  $b$,  $c$,  $d \in G$ 
\begin{itemize}
\item CAI. \textit{Cyclic associativity I}:  $a + (b + c) = c + (a + b)$.
\item CAII. \textit{Cyclic associativity II}:  $a + (b + c) = (c + a) + b$.
\item AGI. \textit{Abel-Grassmann I}:  $a + (b + c) = c + (b + a)$.
\item AGII. \textit{Abel-Grassmann II}:  $a + (b + c) = (b + a) + c$.
\item R. \textit{Reduced product property}: $(a + b) + c =  a + (c + b)$.
\item H. \textit{Hilbert property}\footnote{This property was presented as part of an axiomatization for real numbers in  \cite[p. 51-52]{hilbert}.}: the equations $x + a = b$  and  $a + y = b$  have a unique solution.
\end{itemize}

Algebraic structures whose operations satisfy some of these properties have also received special names such as \textit{Quasigroup}\footnote{This concept was introduced by B. A. Hausmann and O. Ore in 1937 \cite[p. 22]{ilse}. An equivalent definition appears in \cite[p. 50]{warner}.},  a groupoid that satisfy H and \textit{Loop}, a quasigroup having a neutral element. 

\begin{theorem}\label{teor1}
If  $(G, +)$  is a commutative semigroup  then it satisfies the properties CAI, CAII, AGI, AGII and R. 
\end{theorem}

\begin{theorem}\label{teor2}
If   $(G, +)$  is an abelian group then it satisfies H. 
\end{theorem}

Although classical structures such as the abelian group and the commutative semigroup satisfy the properties mentioned, this does not mean that these properties are not independent.

\section*{Examples}

\begin{enumerate}
\item The natural numbers with usual addition and multiplication are commutative semigroups, have neutral elements 0 and 1 respectively, but are not quasigroups.
\item Integers, rational, real and complex numbers with the usual sum are commutative semigroups and loops.
\item A lattice with the meet  ($\land$) and join ($\lor$) operations is a commutative semigroup, but not a quasigroup.
\item The integers with subtraction $x \circ y = x - y$ is a quasigroup and satisfies AGI but not AGII. It neither satisfies ACI nor ACII, R, A or C and also does not  have a neutral element.
\item The integers with reciprocal subtraction $x \bullet y = y - x$ is a quasigroup which satisfies AGII, but not AGI. It neither satisfies ACI nor ACII,  R,  A or C and does not have a neutral element.
\item A set A with the second projection operation defined by $x \ \pi_2 \ y = y$, is a non-commutative semigroup,  which satisfies AGII but not AGI. It neither satisfies ACI nor ACII or R. It has a neutral element and it is not a quasigroup.
\item A set A with the first projection operation defined by $x \ \pi_1 \ y = x$, is a non-commutative semigroup, which satisfies R but not AGI. It neither satisfies AGII nor ACI or ACII. It does not have a neutral element and it is not a quasigroup.
\item In the real interval $[0, 1]$ the operation $p * q = 1 - pq$ is commutative but not associative. It does not have a neutral element. It is neither AGI nor AGII, ACI, ACII or R and it is not a quasigroup. This operation is used in probability theory to determine the probability for two independent events to not occur simultaneously when the probability of occurrence of one is $p$ and of the other is $q$.
\item In the ordered set $\{0, 1/2, 1\}$ the operation defined by table \ref{tabla1}, is commutative, but not associative. It has a neutral element 1. It is not AGI, neither AGII nor ACI, ACII or R and it is not a quasigroup. This operation is the logical equivalence which is used in a trivalent Heyting algebra and it was used by Reichenbach in a formulation of quantum mechanics \cite[366-367]{jammer}.
\begin{table}[h]
\begin{center}
\begin{tabular}{c|ccc}
$\leftrightarrow$&0&1/2&1\\ \hline 
0&1&0&0  \\ 
1/2&0&1&1/2 \\ 
1&0&1/2&1\\
\end{tabular}
\end{center}
\caption{Trivalent logical equivalence} \label{tabla1}
\end{table}
\end{enumerate}

\section{Substituting associativity and commutativity}

A strategy to search for new characterizations of abelian groups is to exchange or replace some of the properties that define them by others such that these new properties, when mixed with the remaining ones, give us a new definition of the abelian group. To this end, we now establish  some relations between structures that satisfy some of the properties  mentioned above. 
 
We can find structures characterized by some of the unusual properties  mentioned above which together with the property NE will result into the properties A and C.   

\begin{theorem}\label{teor3}
If  a groupoid  $(G, +)$  has a neutral element,  $e$, and satisfies the property AGII,  then it is a commutative semigroup.
\end{theorem}

\begin{proof}
We first show that  $+$  is commutative. Applying the properties NE and AGII we obtain
\[a + b  =  a + (b + e) = (b + a) + e = b + a\]

From properties AGII and C we deduce that   $+$  is associative:
\[a + (b + c) = (b + a) + c = (a + b) + c \qedhere\] 
\end{proof}

The proof of theorem \ref{teor4} below is analogous to the one of theorem \ref{teor3}. 

\begin{theorem}\label{teor4}
If  a groupoid  $(G, +)$  has a neutral element and satisfies one of the pro\-perties CAI, CAII, AGI or R, then it is a commutative semigroup. 
\end{theorem}

Note that from theorems \ref{teor3} and \ref{teor4} we can replace the associative and commutative properties by any of the properties CAI, CAII, AGI, AGII or R, in the de\-fi\-ni\-tion of an abelian group. This way we obtain another characterization of the structure under consideration, only with three axioms.

\begin{theorem}
The following conditions are equivalent:
\begin{enumerate}
\item $(G, +)$  is an abelian group.
\item $(G, +)$  is a groupoid that satisfies the properties NE, IN and CAI.
\item $(G, +)$  is a groupoid that satisfies the properties NE, IN and CAII.
\item $(G, +)$  is a groupoid that satisfies the properties NE, IN and AGI.
\item $(G, +)$  is a groupoid that satisfies the properties NE, IN and AGII.
\item $(G, +)$  is a groupoid that satisfies the properties NE, IN and R.
\end{enumerate}
\end{theorem}

It should be noted that the properties CAI, CAII, AGII and AGI have been used \cite[p. 10]{pad} for axiomatizing the lattice theory.

\section{Substituting inverse elements and neutral element}

From theorems \ref{teor1} and \ref{teor2} we deduce that an abelian group satisfies the properties CAI, CAII, AGI, AGII, R and H. 

The next theorem indicates how the property H may be used to characterize the abelian groups, but  without forgetting the commutative and associative properties.







\begin{theorem}\label{teor6}
If   $(G, +)$  is an associative and commutative quasigroup then it is an abelian group.
\end{theorem}

\begin{proof}
Since $(G, +)$ is a quasigroup, for all $a \in G$, the equation  $a + x = a$  has a unique solution, say  $e_a$, i.e.  $a + e_a = a$. By C,  $a + e_a = e_a + a = a$. 

Now, let  $b \in G$  then by A,  $a + b = (a + e_a) + b = a + (e_a + b)$ and as the equation  $a + y = d$  with  $d = a + b$, has a unique solution, we conclude that  $b = e_a + b$. Therefore,  $e_a + b = b = e_b + b$ and again by uniqueness of the solution of the equation  $y + b = b$, it follows that  $e_a = e_b$. Hence,  $e_a$ is the neutral element of  $G$ since the above argument is valid for all  $b \in G$.  

The existence of an inverse element for each element $a$ of $G$ is guaranteed by the exis\-tence of the solution of the equation  $a + x = e$ with  $e$ the neutral element, and pro\-perty C.
\end{proof}

From the arguments presented in the proof of theorem \ref{teor6} we can conclude: 

\begin{theorem}\label{CA}
If   $(G, +)$  is a quasigroup then it satisfies the property of being cancelative (CA). CA is defined as follows: for all  $a, b, c \in G$,
\center{if  \ $a + b = a + c$ \ then \ $a = c$ \ \ and \ \ if  \ $b + a = c + a$ \ then \ $b = c$}
\end{theorem}



Combining the results of theorems \ref{teor2} and \ref{teor6} we obtain other characterizations of abelian groups with three axioms. 

\begin{theorem}
The following conditions are equivalent:
\begin{enumerate}
\item $(G, +)$  is an abelian group.
\item $(G, +)$  is a groupoid that satisfies the properties H, A and C.
\end{enumerate}
\end{theorem}

\section{Substituting all properties}
Below we present results in which we cha\-rac\-te\-ri\-ze the structure of an abelian group without using the usual properties. We focus on replacing the property NE, a key pro\-per\-ty that has been used in previous results, without having to resort to the properties A and C. 

\begin{theorem}\label{CAI}
If   $(G, +)$  is a quasigroup that satisfies the property CAI, then it is a loop.
\end{theorem}

\begin{proof}
For all $a \in G$, let  
\begin{equation}
a + e_a = a
\end{equation} 
with  $e_a$ the unique solution of the equation  $a + x = a$. Combining (1) with the pro\-per\-ty CAI we obtain  $e_a + a = e_a + (a + e_a) = e_a + (e_a + a)$.  From theorem \ref{CA} we get  
\begin{equation}
e_a + a = a.
\end{equation} 

Given  $b \in G$, from (2) and CAI we have  
\[(e_a + b) + a = (e_a + b) + (e_a + a) = e_a + (a + (e_a + b)) = e_a + (b + (a + e_a))\]
and by (1) and CAI we obtain
\[e_a + (b + (a + e_a)) = e_a + (b + a) = b + (a + e_a) = b +a.\]

Then  $(e_a + b) + a = b +a$ and by theorem \ref{CA} we conclude that  $e_a + b = b$. Hence,  $e_a + b = e_b + b$  and again by theorem \ref{CA},  $e_a = e_b$. As this argument is valid for all  $b \in G$ we arrive at the conclusion that $e_a$ is the neutral element of  $G$ which proves the theorem. 
\end{proof}

\begin{theorem}\label{CAII}
If   $(G, +)$  is a quasigroup that satisfies the property CAII, then it is a loop.
\end{theorem}

We shall give two proofs for this theorem: one, similar to the previous proof, sho\-wing directly that there is a neutral element and the other, proving that under these assumptions the property CAI holds and so the assertion follows from theorem \ref{CAI}.  

\begin{proof}[Proof 1]
As for all $a \in G$ the equation $a + x = a$ has a unique solution, say $e_a$, i.e. 
\begin{equation}
a + e_a = a
\end{equation} 
Applying (3) and the property CAII we have  
\[e_a + a = e_a + (a + e_a) = (e_a + e_a) + a\] 
From theorem \ref{CA} we get  
\begin{equation}
e_a + e_a = e_a 
\end{equation} 
Given  $b \in G$, from (4) and CAII it follows that 
\[a + e_a  = a + (e_a + e_a) = (e_a + a) + e_a\]

Again by theorem \ref{CA} we have $a = e_a + a$  for all  $a \in G$. Now let $b \in G$, by (3) and CAII we obtain  $b + a = b + (a + e_a) = (e_a + b) + a$. By theorem \ref{CA},  $b = e_a + b$. Thus,  $e_a + b = e_b + b$ and so  $e_a = e_b$.  As this argument is valid for all  $b \in G$ we again conclude that  $e_a$ is the neutral element of  $G$ and as a consequence  $(G, +)$ is a loop.
\end{proof}

\begin{proof}[Proof 2]
We first show that  $+$  is associative. Let  $a, b, c \in G$, as  $G$  is a quasigroup there is  $u \in G$ such that  $b + u = c$. Hence,  
\begin{equation}
a + (b + c) = a + (b + (b + u))
\end{equation}
Applying the property CAII repeatedly we get
\[a + (b + (b + u)) = ((b + u) + a) + b = (u + (a + b)) + b = (a + b) + (b + u)\] 
and replacing  $b+u$  by  $c$  we conclude that  $a + (b + c) = (a + b) + c$.   

Now let us prove that  $+$  satisfies CAI.  Again applying the property CAII repeatedly to (5), we obtain
\begin{align*}
&a + (b + (b + u)) = a + ((u + b) + b) = (b + a) + (u + b) \\
&\hspace*{0.5cm}= (b + (b + a)) + u = ((a + b) + b) + u = b + (u + (a + b))
\end{align*}
By the property A and replacing  $b+u$  by  $c$,  we have  $a + (b + c) = c + (a + b)$. Then from theorem \ref{CAI} it follows that  $(G, +)$ is a loop. 
\end{proof}

\begin{theorem}\label{AGII}
If   $(G, +)$  is a quasigroup that satisfies the property AGII, then it is a loop.
\end{theorem}

\begin{proof}
For each  $a \in G$ let  $e_a$ the unique solution of the equation  $a + x = a$, i.e. 
\begin{equation}
a + e_a = a
\end{equation}
By (6) and property AGII we have
\[e_a + a = e_a + (a + e_a) = (a + e_a) + e_a = a + e_a = a\]
Therefore, for all  $a \in G$, it holds that
\begin{equation}
e_a + a = a = a + e_a
\end{equation}

Now, let  $b \in G$, then by (7) and the property AGII we obtain  
\[b + a = b + (e_a + a) = (e_a + b) + a\]
Furthermore, by theorem \ref{CA} we conclude that  $b = e_a + b$. Hence,  $e_a + b = e_b + b$ and so  $e_a = e_b$. Since this argument is valid for all  $b \in G$, $e_a$ is the neutral element of  $G$ which proves the theorem. 
\end{proof}

\begin{theorem}\label{R}
If   $(G, +)$  is a quasigroup which satisfies the property R, then it is a loop.
\end{theorem}

\begin{proof}
Since  $G$  is a quasigroup, for each  $a \in G$ there is  $e_a \in G$ such that  
\begin{equation}
a + e_a = a
\end{equation}
Therefore, given any  $b \in G$, by (8) and the property R we get
\[a + b = (a + e_a) + b = a + (b + e_a) \] 
By the theorem \ref{CA} we obtain  $b = b + e_a$. As a consequence  $b + e_a = b + e_b$ and so  $e_a = e_b$. This argument holds for all  $b \in G$. We therefore conclude that there is a unique right neutral element,  which we denote by  $e$. 

On the other hand, for all  $a \in G$  the equation  $y + a = a$  has a unique solution, say  $\hat{e}_a$, i.e. 
\begin{equation}
\hat{e}_a + a = a
\end{equation} 

As  $e$ is a right neutral element we have $(e + a) + e = e + a$. Applying (9) and the pro\-per\-ty R, we obtain  $e + a = e + (\hat{e}_a + a) = (e + a) + \hat{e}_a$. Then  $(e + a) + e = (e + a) + \hat{e}_a$ and by theorem \ref{CA}, $e = \hat{e}_a$ for all  $a \in G$. Thus,  $e$ is also left  neutral element of  $G$ which completes the proof.   
\end{proof}

Finally, the results of theorems \ref{CAI}-\ref{R} together with the theorems \ref{teor1}-\ref{teor4} and \ref{teor6} can be condensed into the following final result.

\begin{theorem}
The following conditions are equivalent:
\begin{enumerate}
\item $(G, +)$  is an abelian group.
\item $(G, +)$  is a groupoid that satisfies the properties H and CAI.
\item $(G, +)$  is a groupoid that satisfies the properties H and CAII.
\item $(G, +)$  is a groupoid that satisfies the properties H and AGII.
\item $(G, +)$  is a groupoid that satisfies the properties H and R.
\end{enumerate}
\end{theorem}

Examples 1 and 4 show the independence of each of the properties CAI, CAII, AGII and R  with respect to property H, and examples 1 and 5 show the independence of AGI and H.

One would think that with AGI we would have an analogous theorem to those presented in this section, however this combination of properties does not even provide us a loop. For example, integers with subtraction satisfy H and AGI but it is not an abelian group as it does not have a neutral element and hence no inverses.

\end{document}